\documentclass{amsart}

\usepackage{amssymb}

\usepackage{graphicx}

\usepackage{pictex}
\usepackage{amsmath,amsthm,amssymb,pictex,verbatim, cite}

\newtheorem{theorem}{Theorem}[section]
\newtheorem{lemma}[theorem]{Lemma}

\newtheorem{corollary}[theorem]{Corollary}

\theoremstyle{remark}

\newtheorem{example}[theorem]{Example}

\numberwithin{equation}{section}

\begin{document}

\title{Voting for Committees in Agreeable Societies}

\author{Matt Davis}
\address{Department of Mathematics, Muskingum University, New Concord, OH  43762}
\email{mattd@muskingum.edu}

\author{Michael E. Orrison}
\address{Department of Mathematics, Harvey Mudd College, Claremont, CA  91711}
\email{orrison@math.hmc.edu}

\author{Francis Edward Su}
\address{Department of Mathematics, Harvey Mudd College, Claremont, CA  91711}
\email{su@math.hmc.edu}
\thanks{Francis Edward Su was supported in part by NSF Grant DMS-1002938.}


\date{}

\begin{abstract} 
We examine the following voting situation. A committee of $k$ people is to be formed from a pool of $n$ candidates. The voters selecting the committee will submit a list of $j$ candidates that they would prefer to be on the committee. We assume that $j \le k \le n$. For a chosen committee, a given voter is said to be satisfied by that committee if her submitted list of $j$ candidates is a subset of that committee. We examine how popular is the most popular committee.  In particular, we show there is always a committee that satisfies a certain fraction of the voters and examine what characteristics of the voter data will increase that fraction.
\end{abstract}

\maketitle

\section{Introduction}

The goal of this article is to examine the following voting situation: from a pool of $n$ candidates, a committee of size $k$ is to be formed.  Every voter will submit an unordered list of $j$ candidates that they would prefer be on the committee.  We will assume throughout that $j \leq k < n$.   

Once this voter data is collected, a committee will be chosen via some procedure.  One way to select a winner of the election is to look for a most ``popular'' committee: a committee with the highest proportion of voter ``approval''.  (In case of ties, there may be more than one most popular committee.)  We must make this notion precise.

Let us say that a voter \emph{approves} a committee if it contains the list of $j$ candidates that the voter submitted.  Note that a voter may approve several committees, since several different $k$-element subsets will contain a given $j$-element subset if $j < k$.  

We wish to answer questions regarding whether it is always possible, assuming certain conditions on the voters' preferences, to find a committee approved by a certain proportion of voters.  Phrased another way, we can ask how popular is a most popular (most-approved) committee?  Can we make any minimal guarantees for the approval proportion of a most popular committee, given some geometric condition on the voter preferences?

Such questions have been addressed in voting contexts where each person is choosing just one candidate.  Berg et~al.~\cite{berg} initiated the study of \emph{agreeability conditions} for voting preferences over a linear political spectrum.  In this situation, candidates are chosen from a line of possibilities, and each person is allowed to specify an interval of candidates that they find acceptable.  In other words, their \emph{approval set} is an interval in the real line.  They call a society \emph{super-agreeable} if every pair of voters has a mutually agreed candidate (i.e., their approval sets overlap) and they also consider other \emph{agreeability} conditions that specify ways in which voter preferences, when taken in small subsets, are locally similar.  They show how these produce global conclusions about how popular a most popular candidate must be, and give minimal guarantees for how many approval sets a most-approved candidate must lie in.

These ideas have been generalized to other kinds of geometric spaces that can be viewed as political spectra---circles \cite{hardin}, trees \cite{fletcher}, multi-dimensional spaces \cite{eschenfeldt}---as well as other kinds of approval sets \cite{nyman} and agreeability conditions \cite{carlson-oneill}.

This paper considers similar questions in a new context: voting for committees.  (See  \cite{brams-etal, fishburn, fishburn-pekec, gehrlein, kilgour-marshall, ratliff-2003, ratliff-2006} for several different approaches to the study of voting for committees.)  The set of $n$ candidates is finite, and not viewed as a geometric space.  However, voters are now allowed to specify a list of size $j$, and there is a geometric notion of how close two lists may be.  The resulting geometric space of lists is called a \emph{Johnson graph} $J(n,j)$ where each point is a $j$-element subset of an $n$-element set (see \cite{brouwer-etal} for more on Johnson graphs).  

Unlike the context of Berg~et~al.~\cite{berg}, here the space in which preferences are expressed is \emph{different} from the space of potential outcomes of an election, which are the $k$-element subsets of the candidates, and which can be thought of in terms of the graph $J(n,k)$.  In particular, each voter's list in $J(n,j)$ produces an approval set in $J(n,k)$ consisting of all committees that contain the list.  In this paper, we ask: what is the minimal guarantee for the popularity of a most popular committee in $J(n,k)$, and if we place conditions on how similar, or ``agreeable'', the submitted lists are in $J(n,j)$, how does that change the guarantee? 

We first demonstrate, in Theorem \ref{firstprop}, a sharp lower bound for any voter distribution.  We then show, in Theorem \ref{theorem}, that if all the votes lie within a ``ball'' in the space of lists, then the bound we can guarantee improves.  We conclude with some extensions and open questions.

\section{Definitions}

Let us identify the candidates with the elements of the set $[n] = \{1,2, \ldots n\}$.  Any $j$-element subset of 
$[n]$ will be called a \emph{list}.  (Note that a list here is just a set, and is unordered.) 
Any $k$-element subset of $[n]$ will be called a \textit{committee}.

We choose to think of given voter data as a probability distribution $P$ on the set of lists, which we call the \emph{voter distribution}.  In particular, this distribution will specify for each list $\ell$ the \emph{voting proportion} $P(\ell)$---this is the fraction of voters who submitted the list $\ell$.  
We shall use the same notation to describe the probability of a collection of lists, so that $P(\mathcal{A})$ describes the likelihood that a voter chose one of the lists in the collection $\mathcal{A}$.

If $C$ is a committee, let  
$$\pi_{P}( C) = \sum_{\ell \subseteq C, |\ell|=j} P(\ell)$$
denote the \emph{approval proportion} of committee $C$ with respect to voter data $P$: this is the fraction of voters that approve a given committee if the voter distribution on the set of all lists is $P$.  We will write simply $\pi( C)$ if $P$ is understood.

\begin{example} Assume that candidates 1 through 7 are being considered for a 4-person committee, and voters are asked to submit lists of three candidates. Define $P$ so that $P(\{1,2,3\}) = 7/15$, while $P(\ell) = 2/15$ for the lists $\{4,5,6\}$, $\{4,5,7\}$, $\{4,6,7\}$, and $\{5,6,7\}$. (Then $P(\ell) = 0$ for all other lists.) For this choice of $P$, the list $\{1,2,3\}$ has the highest voting proportion, but by inspection, we see that $\pi_{P}(\{4,5,6,7\}) = 8/15$, which is a higher approval proportion than any other committee.

This example shows that a most popular committee is not necessarily populated by candidates who are most popular. Even though each of candidates 1 through 3 appears in $7/15$ of the submitted lists, and candidates 4 through 7 each appear in only $6/15$ of the submitted lists, the committee $\{4,5,6,7\}$ is still the most popular.
\end{example}

Our first observation is straightforward; it gives a minimum popularity for a most popular committee.  The proof also illustrates a technique that we will exploit again later---that a most popular committee has at least as large an approval proportion than the \emph{average} approval proportion over all committees.

\begin{theorem}
\label{firstprop}
For any given voter distribution $P$, there exists a committee $\widehat C$ with approval proportion satisfying
$$ \pi(\widehat C ) \geq \frac{\binom{k}{j}}{\binom{n}{j}}.$$
\end{theorem}

\begin{proof}
If we sum the approval proportions over all $\binom{n}{k}$ possible committees for $C$, we obtain
\begin{eqnarray*}
\sum_{C \subseteq [n], |C|=k} \pi( C) 
&=& \sum_{C \subseteq [n], |C|=k}   \left( \sum_{\ell \subseteq C, |\ell|=j} P(\ell) \right) \\
&=& \binom{n-j}{k-j} \sum_{\ell \subseteq [n], |\ell| = j} P(\ell) \\
&=& \binom{n-j}{k-j}.
\end{eqnarray*}
The first equality follows from the definition of the approval proportion.
The second equality follows by noting that each list $\ell$ will be satisfied by the $\binom{n-j}{k-j}$ committees that include $\ell$ as a subset; hence the term $P(\ell)$ will appear $\binom{n-j}{k-j}$ times if we sum over lists, then committees.
And the third equality is a consequence of $P$ being a probability distribution.

Since there are $\binom{n}{k}$ terms in the sum, at least one committee $\widehat{C}$ must have approval proportion $\pi(\widehat{C})$ at least as large as the average:
$$
\pi(\widehat{C}) \geq \frac{\binom{n-j}{k-j}}{\binom{n}{k}} 
= \frac{(n-j)!k!}{(k-j)!n!} 
=  \frac{\binom{k}{j}}{\binom{n}{j}}.
$$
\end{proof}

Notice that this bound is the best possible, since it is achieved by, for example, the uniform distribution on the set of all lists.  

\section{Votes Within a Ball}

We now investigate what characteristics of the voter distribution might improve the popularity of a most approved committee.  For instance, in Theorem \ref{firstprop}, the worst possible case is one in which the distribution is ``spread out'' over all possible lists uniformly.  But if the voters are ``agreeable'' in some sense then we might be able to guarantee a committee that will be approved by more voters.  As an example, if we knew that all of the voters' lists were ``close'', then this would suggest that the approval proportion of a most popular committee should be higher than the bound in Theorem \ref{firstprop}.  How can we can describe the ``closeness'' of votes?  

Consider the \emph{Johnson graph} $J(n,j)$, whose vertices are the $j$-sets (i.e., $j$-element subsets) of $[n]$, with two $j$-sets $v$ and $w$ adjacent if they have exactly $j-1$ elements in common.  In our context, we think of the Johnson graph as the space of possible lists, with two lists adjacent if they differ by exactly one candidate.  Then note that the graph distance $d$ in $J(n,j)$  has a nice interpretation: for two lists $v$ and $w$, the distance $d(v,w)=m$ if and only if $|v \cap w| = j-m$, i.e., $v$ and $w$ differ in exactly $m$ places.  This graph has diameter
$D = \textrm{min}(j,n-j)$, which is achieved when $j$-sets are as disjoint as possible.

\begin{figure}
\includegraphics{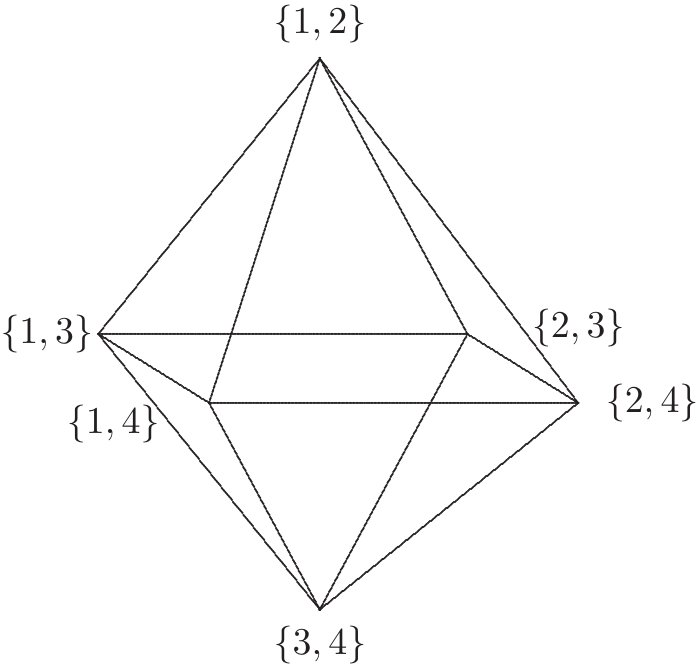}
\caption{The Johnson graph $J(4,2)$.} 
\label{fig:JohnsonGraph} 
\end{figure}


We will explore how to improve the bound in Theorem \ref{firstprop} if all the submitted lists are within some ``ball'' in the Johnson graph about a central fixed list $v$.  Given a non-negative integer $r$, we define the \textit{ball of radius $r$ around $v$} to be 
$$
B_{r}(v) = \{ w \textrm{ a vertex of } J(n,j) \, \mid \, d(v,w) \leq r\}.
$$
Then $B_{r}(v)$ is the entire Johnson graph if and only if $r \geq D$.  In what follows, we assume that $j >1$ so that the Johnson graph has diameter at least 2, and neighborhoods of vertices can be proper subsets of the set of vertices.

Figure \ref{fig:JohnsonGraph}, though a small example, gives some helpful intuition concerning the structure of Johnson graphs. We will examine \emph{rings} of lists in $J(n,j)$ that are equidistant from $v$. 
More formally, define the \emph{ring} $R_{r}(v)$ by
 $$
 R_{r}(v) = \{ x \in J(n,j) \, | \, d(x,v) = r \},
 $$
for $0 \leq r \leq D$, which may be thought of as the ring of radius $r$ about the vertex $v$.  We note that if $d(x,v)=r$, then the list $x$ must have $j-r$ candidates in common with the list $v$, and $r$ candidates different from $v$.  Thus
\begin{equation}
\label{Rrv}
|R_{r}(v)| = \binom{j}{j-r} \binom{n-j}{r}.
\end{equation}
For example, in Figure \ref{fig:JohnsonGraph}, if $v = \{1,2\}$, the rings are:
\begin{eqnarray*}
R_{0}(v) &=& \{ \{1,2\} \},\\
R_{1}(v) &=& \{ \{1,3\}, \{1,4\}, \{2,3\}, \{2,4\}\},\\
R_{2}(v) &=& \{ \{3,4\} \}.
\end{eqnarray*}

\begin{example} \label{ex:betterbound} Let $n=6$, $k=4$, and $j=3$.  Theorem \ref{firstprop} implies that for any $P$, there will always be a committee $\widehat{C}$ satisfying $\pi(\widehat{C}) \geq \frac{1}{5}$.  Now let $v = \{1,2,3\}$ and assume that $P$ is supported on the ball $B_{1}(v)$, with $p = P(v)$. Then the complementary probability is \[ 1-p = \sum_{w \in R_{1}(v)} P(w) = \sum_{i=4}^{6} P(\{1,2,i\}) + P(\{1,3,i\}) + P(\{2,3,i\}). \] Hence, since one of the summands must be at least as large as the average, 
\[ P(\{1,2,i\}) + P(\{1,3,i\}) + P(\{2,3,i\}) \geq \frac{1}{3}(1-p),\]
for some $i$ with $4 \leq i \leq 6$.  Then for $C = \{1,2,3,i\}$, we have 
\begin{equation} 
\pi_{P}(C) = P(v) + P(\{1,2,i\}) + P(\{1,3,i\}) + P(\{2,3,i\}) \geq p + \frac{1}{3}(1-p) \geq \frac{1}{3}. \label{eq:thmexample} 
\end{equation} 
In fact, if $p=0$ and $P$ is distributed uniformly on $R_{1}(v)$, then the inequality in \eqref{eq:thmexample} is equality, and $\frac{1}{3}$ is the highest approval proportion for any committee. So in this case, with the added assumption that $P$ is supported on $B_{1}(v)$, we can improve the lower bound on $\pi_{P}(\hat{C})$ to $1/3$. (This will also be a consequence of Theorem \ref{theorem} below.)
\end{example}

Rings about a vertex $v$ always follow the same pattern: 
as $r$ increases from 0 to $D$, the rings increase in size monotonically for a time, 
then decrease monotonically for a time. In particular, we can show:
\begin{lemma}
$ |R_{r}(v)| \leq |R_{r+1}(v)|$ holds if and only if $$\quad r \leq \frac{nj-j^{2}-1}{n+2}.$$
\end{lemma}
\begin{proof}
A straightforward calculation using (\ref{Rrv}) shows that both conditions are equivalent to the condition 
$(j-r)(n-j-r) \geq (r+1)(r+1)$.
\end{proof}
It is helpful to visualize the Johnson graph as points on a sphere, with the list $v$ at the north pole, and the rings $R_{r}(v)$ drawn as latitude lines. (This visualization explains our terminology and captures the behavior of the relative sizes of the rings.)

Suppose the voter distribution $P$ is supported on a neighborhood $B_{\rho}(v)$ of radius $\rho$.  To find a minimal guaranteed approval proportion, we are interested in the worst-case scenario -- the distribution that leads to a worst possible ``best'' committee. Intuitively, we would expect this worst-case scenario to occur when the voters' lists are spread out as far away from $v$ as possible, i.e., on the ring $R_{\rho}(v)$. However, if $\rho$ is too large, the outermost ring $R_{\rho}(v)$ will be too small, and votes on that ring will be too close together. In that case, the worst-case scenario will result from a more complicated distribution of voters. 

\begin{example}
 \label{ex:ringexample} 
Let $n=6$, $k=4$, and $j=3$, and let $v = \{1,2,3\}$. In Example \ref{ex:betterbound}, we showed that if the voter distribution $P$ is supported on $B_{1}(v)$, then there must be a committee $C$ with $\pi_{P}(C) \geq 1/3$. As stated there, if $P$ is in fact the uniform distribution on $R_{1}(v)$, then every committee $C$ that contains all of $v$ has $\pi_{P}(C) = 1/3$, and $1/3$ is the maximum approval proportion in this case. Moreover, if $P(v) > 0$, then the discussion in Example \ref{ex:betterbound} shows that one of the committees that contains $v$ will have an approval proportion strictly larger than $1/3$. Thus the ``worst-case'' scenario occurs only when $P$ is supported on the outermost ring $R_{1}(v)$. 

Conversely, if $P$ were supported on $B_{2}(v)$, then the worst-case scenario cannot result from $P$ being supported only on $R_{2}(v)$. Because $R_{2}(v) = R_{1}(\{4,5,6\})$, the previous example shows that at least one committee $C$ would have $\pi_{P}(C) \geq 1/3$ if $P$ were supported on $R_{2}(v)$. However, if $P$ is the uniform distribution on $B_{2}(v)$, then it may be checked that the largest approval proportion for any committee is $4/19$. In general, we expect that the worst-case scenario will occur when $P$ is supported on $R_{\rho}(v)$ only when $\rho$ is small. 
\end{example}

\section{Concentric Voter Distributions}

We begin by showing that the rings, aside from being helpful in visualizing the neighborhood $B_{\rho}(v)$, are also quite important to our analysis.  We define a \textit{concentric distribution centered at $v$} to be a distribution $P$ such that, for any $r$ and for any two lists $\ell_{1}$ and $\ell_{2}$ in $R_{r}(v)$, the proportion of voters choosing $\ell_{1}$ and $\ell_{2}$ are the same: $P(\ell_{1})=P(\ell_{2})$.  That is, the votes for lists in a given ring are distributed uniformly on that ring.  

Let $w_r$ denote the \emph{weight} of the ring $R_{r}(v)$: it is the sum of the voting proportions of elements of the ring $R_{r}(v)$.  
Then a concentric distribution is completely determined by the weights $w_r$, and these weights sum to 1.

\begin{lemma} \label{lem:ringlemma} For any voting data $P$, let $P^{\circ}$ represent the concentric distribution centered at $v$ that has the same weights as $P$.   Then if $C$ is a committee with the highest approval proportion $\pi_{P}(C )$, and $C^{\circ}$ is a committee with the highest approval proportion $\pi_{P^{\circ}}(C^{\circ})$, then 
$$ \pi_{P}(C ) \geq \pi_{P^{\circ}}(C^{\circ} ).$$
\end{lemma}

In other words, uniformizing the voter distribution concentrically over the rings centered at $v$ can only decrease the popularity of a most popular committee.

\begin{proof}
We will treat the weights as fixed for now, and partition the set of committees into classes by the number of elements a committee differs from the central list $v$.  We'll show that the approval proportion $\pi_{P}$ dominates $\pi_{P^{\circ}}$ in each class.

Let $\mathcal{C}_{m}$ denote the set of committees that differ from the list $v$ in exactly $m$ candidates, i.e., $m$ candidates are ``missing''. 
Thus
$$
\mathcal{C}_{m} = \{ C \subseteq [n] \, : \, |C| = k, |C \cap v| = j-m\}.
$$
For instance, $\mathcal{C}_{0}$ consists of all committees that contain the list $v$.

Note that $m \leq n-k$, because any $k$-person committee in $\mathcal{C}_{m}$ is missing $n-k$ candidates, 
including by construction exactly $m$ candidates in $v$.  We also see $m \leq j$, since $v$ has only $j$ candidates.
Every committee belongs to exactly one $\mathcal{C}_{m}$ for some $m$, and there are $\binom{j}{j-m} \binom{n-j}{k+m-j}$ committees in $\mathcal{C}_{m}$.

\begin{figure}
\includegraphics{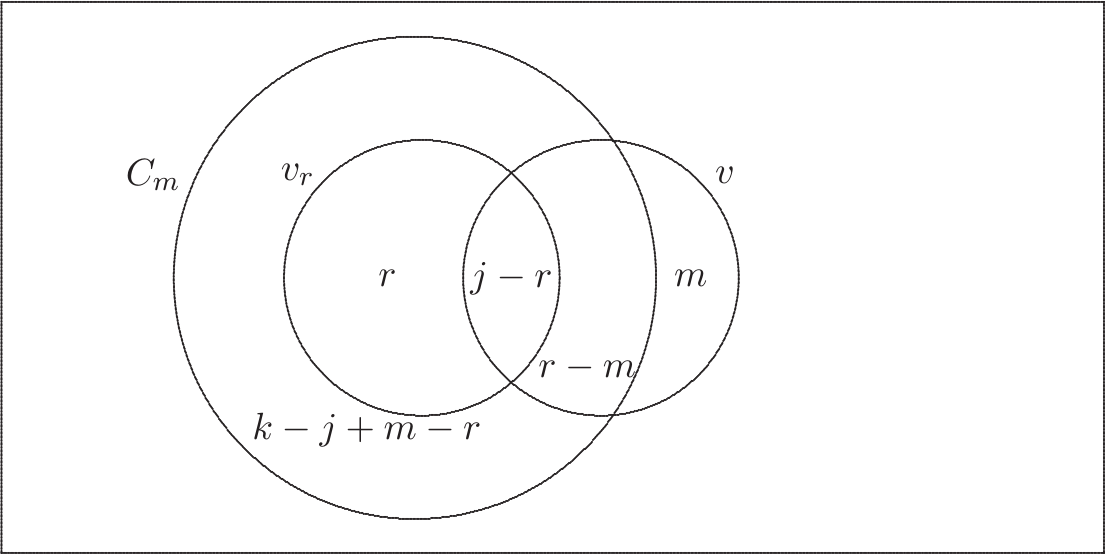}
\caption{Committee $C_{m}$ satisfies list $v_{r}$, which is distance $r$ from list $v$.}
\label{fig:vrdiagram}
\end{figure}


We can count the number of committees in $\mathcal{C}_{m}$ that contain a fixed list $v_{r}$ in the ring $R_{r}(v)$.
A list $v_{r} \in R_{r}(v)$ that is contained in a committee $C \in \mathcal{C}_{m}$ can be visualized as in Figure $\ref{fig:vrdiagram}$, which is labeled with the sizes of the various subsets involved.
To construct such a committee $C$, we must choose $r-m$ elements from the $r$ elements in $v \setminus v_r$, and $k - j + m -r$ elements from the $n-j-r$ elements not in $v \cup v_{r}$. Thus there are $\binom{r}{r-m} \binom{n-j-r}{k-j+m-r}$ such committees.

Then consider:
$$\sum_{C \in \mathcal{C}_{m}} \pi_{P}( C)
= \sum_{C \in \mathcal{C}_{m}} \sum_{\ell \subseteq C, |\ell|=j} P(\ell).
$$
Note that $P(\ell)$ appears multiple times in this sum, once for every instance where a committee $C$ in $\mathcal{C}_{m}$ contains $\ell$.  The number of times that happens depends on what ring $\ell$ is in; for $\ell$ in $R_{r}(v)$, the number of committees in $\mathcal{C}_{m}$ that contain $\ell$ is $\binom{r}{r-m} \binom{n-j-r}{k-j+m-r}$ as calculated above.  So our sum $\sum_{C \in \mathcal{C}_{m}} \pi_{P}( C)$ now becomes:
\begin{eqnarray*}
\sum_{C \in \mathcal{C}_{m}} \pi_{P}( C)
&=& 
\sum_{r=m}^{D} \left[ \sum_{\ell \in R_{r}(v)} P(\ell) \right] \binom{r}{r-m} \binom{n-j-r}{k+m-j-r}\\
&=& \sum_{r=m}^{D} \quad  w_r \quad \binom{r}{r-m} \binom{n-j-r}{k+m-j-r},
\end{eqnarray*}
where $w_r$ is the weight of the ring $R_{r}(v)$.  The sum index starts at $r=m$ because all subsets of $\mathcal{C}_{m}$ by construction must differ from $v$ by at least $m$ elements, and the index ends at $D$, the diameter of the graph.

Then at least one of the committees $\widehat C \in \mathcal{C}_{m}$ must have approval proportion equal to or better than average, or in other words, 
\begin{eqnarray}
\pi_{P}(\widehat C) 
&\geq& 
\frac{1}{ \binom{j}{m} \binom{n-j}{k+m-j}}
\sum_{r=m}^{D} w_r \binom{r}{r-m} \binom{n-j-r}{k+m-j-r}  \nonumber\\
&=& \sum_{r=m}^{D} w_r \frac{r!}{(r-m)!} \frac{(n-j-r)!}{(k+m-j-r)!} \frac{(j-m)!}{j!} \frac{(k+m-j)!}{(n-j)!}. 
\label{eq:committeearbmin} 
\end{eqnarray}

Now, consider a similar analysis using concentrically distributed voter data $P^{\circ}$ centered at $v$ with the same ring weights $w_{r}$ as $P$.  Specifically, we assume that each list in $R_{r}(v)$ receives voting proportion 
$$\frac{w_r}{\binom{j}{j-r} \binom{n-j}{r}}.$$
Then, we note that a committee $C^{\circ}$ in $\mathcal{C}_{m}$ will contain $\binom{j-m}{j-r}\binom{k+m-j}{r}$ of the lists in the ring $R_{r}(v)$. 
Again referring to Figure \ref{fig:vrdiagram}, we see that a list that is in $R_{r}(v)$ must contain $j-r$ elements of $v$ chosen from the $j-m$ such elements in $C^{\circ}$, and must contain $r$ candidates not in $v$ chosen from the $k+m-j$ such candidates in $C^{\circ}$.

Thus, if $C^{\circ} \in \mathcal{C}_{m}$, we have
\begin{eqnarray}
\pi_{P^{\circ}}(C^{\circ}) 
&=& \sum_{r=m}^{D} w_r \frac{\binom{j-m}{j-r}\binom{k+m-j}{r}}{\binom{j}{r} \binom{n-j}{r}} \nonumber \\
&=& \sum_{r=m}^{D} w_r \frac{(j-m)!}{(r-m)!} \frac{(k+m-j)!}{(k+m-j-r)!} \frac{r!}{j!} \frac{(n-j-r)!}{(n-j)!}. \label{eq:committeetotalpm}
\end{eqnarray}
So every committee in $\mathcal{C}_{m}$ has the same approval proportion under $P^{\circ}$.  
Comparing the final expressions in $\eqref{eq:committeearbmin}$ and $\eqref{eq:committeetotalpm}$, we see that both are equal to
$$\sum_{r=m}^{D} w_r b_{r,m}$$
where 
$$
b_{r,m} = \frac{r!}{(r-m)!} \frac{(k+m-j)!}{(k+m-j-r)!} \frac{(j-m)!}{j!} \frac{(n-j-r)!}{(n-j)!}.
$$
Thus 
$$\pi_{P}(\widehat C) \geq \pi_{P^{\circ}}(C^{\circ}).$$

That is, the approval proportion of a most popular committee in $\mathcal{C}_{m}$ under $P$ will equal or exceed the approval proportion of any  committee in $\mathcal{C}_{m}$ under $P^{\circ}$.  Then, taking the maximum over all $m$, we obtain the desired result.
\end{proof}

The numbers $b_{r,m}$ are integral to comparing the number of lists contained in various committees. Most important are their relative sizes for a fixed value of $r$.

\begin{lemma} We have $r \leq j \left[ 1 - \frac{j-m}{k+1}\right]$ (using the notation of Lemma \ref{lem:ringlemma}) if and only if
$b_{r,m} \geq b_{r,m+1}$. 
\label{lem:ineqlemma} 
\end{lemma}

As an example of this, if $m=0$ and $j$ is approximately half of $k$, then we obtain the condition that $r$ is less than half of $j$.

\begin{proof}
We start with the inequality $b_{r,m} \geq b_{r,m+1}$. After expanding the binomial coefficients and canceling common terms, we see this is equivalent to
$$
(j-m)(k+m+1-j-r) \geq (r-m)(k+m+1-j)
$$
which is equivalent to $r \leq \frac{j(k+1 +m-j)}{k+1}$, as desired.
\end{proof}

\section{Main Theorem}
We are now in a position to prove our main theorem.  We now assume that the voter distribution $P$ is supported in some ball $B_{\rho}(v)$ centered at a list $v$, with $\rho < D$ so that $B_{\rho}(v)$ is not the entire Johnson graph.

\begin{theorem}
\label{theorem} 
Let $v = \{1,2, \ldots j\}$ and consider a voter distribution $P$ supported on $B_{\rho}(v)$ in the Johnson graph.  
If $\rho \leq j \left[ 1 - \frac{j}{k+1}\right]$, then there exists a committee that satisfies at least $\frac{\binom{k-j}{\rho}}{\binom{n-j}{\rho}}$ of the voters. \end{theorem}

\begin{proof}
Since we are trying to minimize the proportion of voters satisfied by a most popular committee, 
Lemma \ref{lem:ringlemma} shows that we may as well assume the voting distribution is concentric, which we do from now on. 
Now, for each $m$, let $C_{m}$ be any committee in $\mathcal{C}_{m}$. Then the expression given in \eqref{eq:committeetotalpm} gives the proportion of voters satisfied by this committee in terms of the $w_r$. The maximum over all $m$ is an optimal committee for the given profile, so our goal is to minimize that maximum value. 
But, under our assumptions, we see that
$$
r \leq \rho \leq j \left[ 1 - \frac{j}{k+1}\right] \leq \frac{j(k+1+m-j)}{k+1}.
$$

Then Lemma \ref{lem:ineqlemma} shows that $b_{r,m} \geq b_{r,m+1}$. Thus, under our assumptions for any fixed values of $w_r$, we always have 
$$
\sum_{r=m}^{\rho} w_r b_{r,m} \geq \sum_{r=m}^{\rho} w_r b_{r,m+1},
$$
since the weights $w_r$ are non-negative. Thus we see that $C_{m}$ is more widely approved of than $C_{m+1}$, and $C_{0}$ is always a most-approved committee.

Then we have to choose the $w_r$ to minimize the value of $\sum_{r=0}^{\rho} w_rb_{r,0}$. But
\begin{eqnarray*}
b_{r,0} 
&=& \frac{(k-j)!}{(k-j-r)!} \frac{(n-j-r)!}{(n-j)!}\\
&=& \frac{(k-j)!}{(n-j)!} \cdot (n-j-r)(n-j-r-1) \ldots (k-j-r+1).
\end{eqnarray*}
But there are always $n-k$ terms in $(n-j-r)(n-j-r-1) \ldots (k-j-r+1)$, beginning at $n-j-r$. Thus, $b_{0,0} \geq b_{1,0} \geq \ldots$, and the minimum value of $\pi(C_{0}) = \sum_{r=0}^{\rho} w_r b_{r,0}$ is achieved by letting $w_{\rho} = 1$, and letting every other $w_r$ be $0$. In this case, the expression in \eqref{eq:committeetotalpm} gives the desired lower bound. 
\end{proof}

\section{Extensions}

There are some natural extensions of Theorem \ref{theorem} to other situations that may be of interest.
The first concerns what we can guarantee if only a fraction of the voters submit votes in a neighborhood $B_\rho(v)$.

\begin{corollary}
Let $\rho \leq j \left[ 1 - \frac{j}{k+1}\right]$.  Suppose $P$ is a voter distribution in which a proportion $\alpha$ of the voters select lists in a ball $B_{\rho}(v)$, e.g., $P(B_{\rho}(v))=\alpha$.  Then there exists a committee $C$ with approval proportion
$$\pi_{P}( C) \geq \frac{ \binom{k-j}{\rho}}{\binom{n-j}{\rho}} \cdot \alpha.$$
\end{corollary}

\begin{proof}
This follows from applying Theorem \ref{theorem} to just those voters whose votes lie in $B_\rho(v)$.  
We see there exists a committee
that satisfies $\frac{\binom{k-j}{\rho}}{\binom{n-j}{\rho}}$ of those voters, which make up $\frac{ \binom{k-j}{\rho}}{\binom{n-j}{\rho}} \cdot \alpha$ of the entire set of voters.
\end{proof}

We can also consider slightly more general election procedures.  Suppose in choosing a $k$-member committee from a pool of $n$ candidates, that we allow voters to submit a smaller unranked list between $1$ and $j$ candidates that they would prefer to have on the committee.  If these votes are sufficiently similar to $v$ in a sense that we make precise below, then we can obtain a result like (and because of) Theorem \ref{theorem}.

\begin{corollary}
Consider a distribution of votes over lists of size less than or equal to $j$, in such a way that each submitted list is a subset of some list in a ball $B_{\rho}(v)$ of radius $\rho \leq j \left[ 1 - \frac{j}{k+1}\right]$ in the Johnson graph.  Then there exists a committee that satisfies at least $\frac{\binom{k-j}{\rho}}{\binom{n-j}{\rho}}$ of the voters. 
\end{corollary}

\begin{proof}
We can convert such a distribution on subsets of size $\leq j$ into one on $j$-sets satisfying the assumptions in Theorem \ref{theorem}, in a way that can only decrease the number of satisfied voters. To see this, consider any voter who submits a list $v$ containing fewer than $j$ candidates. We replace that list with a list $v'$ in $B_{\rho}(v)$ that contains $v$ as a subset. (Note that a voter who submitted $v'$ would be satisfied by strictly fewer committees than a voter who submitted the list $v$.) Then Theorem \ref{theorem} will apply to the resulting voting distribution, and because our alterations might only have decreased the number of satisfied voters, the lower bound still holds.
\end{proof}

One may also consider elections using \emph{thresholds} as described in \cite{fishburn-pekec}, and ask whether our methods would extend in that context, in which a voter submitting a list $\ell$ approves a committee $C$ if $C$ contains sufficiently many members of $\ell$.  Our main result is the special case that each voter only approves committees that contain \emph{all} $j$ of the candidates from their list.
If voters approve committees that contain at least $s$ of the candidates from their list, for some $s$ between $0$ and $j$, then the analogous result to Lemma \ref{lem:ringlemma} holds and can be proved in a similar way.  However, the analogues of the numbers $b_{r,m}$ are significantly more complicated, making the development of a result similar to Lemma \ref{lem:ineqlemma} a barrier to proving a generalization of the main theorem.



\end{document}